\documentclass[12pt,reqno]{amsart}

\usepackage{amssymb,amsmath, amsfonts, latexsym}
\usepackage{enumerate}
\newtheorem{thm}{Theorem}[section]

\newtheorem{lem}[thm]{Lemma}
\newtheorem{prop}[thm]{Proposition}

\newtheorem{rem}{Remark}[section]

\def\build#1_#2^#3{\mathrel{\mathop{\kern 0pt#1}\limits_{#2}^{#3}}}

 \numberwithin{equation}{section}

\begin{document}

\title[MDP for some estimators of the  realised (co-)volatility]{Moderate deviations for  bipower variation of general function and Hayashi-Yoshida estimators}

\author{Hac\`ene Djellout}
\email{Hacene.Djellout@math.univ-bpclermont.fr}
\address{Laboratoire de Math\'ematiques, CNRS UMR 6620, Universit\'e Blaise Pascal, Campus universitaire des C\'ezeaux, 3 place Vasarely, TSA 60026, CS 60026, 63178 AUBI\`ERE CEDEX, FRANCE.} 

\author{Arnaud Guillin}
\email{Arnaud.Guillin@math.univ-bpclermont.fr}
\address{Laboratoire de Math\'ematiques, CNRS UMR 6620, Universit\'e Blaise Pascal, Campus universitaire des C\'ezeaux, 3 place Vasarely, TSA 60026, CS 60026, 63178 AUBI\`ERE CEDEX, FRANCE.} 

\author{Hui  Jiang}
\email{huijiang@nuaa.edu.cn}
\address{Department of Mathematics, Nanjing University of Aeronautics and Astronautics,
29 Yudao Street, Nanjing 210016, China.}

\author{Yacouba Samoura}
\email{Yacouba.Samoura@math.univ-bpclermont.fr}
\address{Laboratoire de Math\'ematiques, CNRS UMR 6620, Universit\'e Blaise Pascal, Campus universitaire des C\'ezeaux, 3 place Vasarely, TSA 60026, CS 60026, 63178 AUBI\`ERE CEDEX, FRANCE.}


\date{\today}
\begin{abstract}  We consider the moderate deviations behaviors for two (co-) volatility estimators: generalised bipower variation,  Hayashi-Yoshida estimator. The  results are obtained by using a new result about the moderate deviations principle for $m$-dependent random variables based on the Chen-Ledoux type condition. 

\end{abstract}

\maketitle

\vspace{-0.5cm}

\begin{center}
\end{center}

\medskip

\section{Motivation and context }

The principle of moderate deviations (PDM, in short) is a subject of classic study of the probability theory. Indeed, in the study of the limit theorems  of a probability or statistical model, the PDM is one of main questions that we look, after the laws of large numbers, the central limit theorem (CLT, in short)  and the law of the iterated logarithm.

\medskip
The MDP can be seen as an intermediate behavior between the CLT and large
deviations principle (LDP, in short). Usually, the MDP exhibit a simpler rate function (quadratic) inherited from the approximated Gaussian process, and holds for a larger class of dependent random variables than the LDP. 

\medskip

The LDP and  MDP of sums of random variables is now a wide and fastly growing branch of probability theory. It was created initially in the framework of the theory of sums of independent identically distributed random variables and then extended to a wide class of random processes, i.e., random functions in one variable, with some general conditions of weak dependence traditional for the theory of random processes.

\medskip

It is best to think of a specific example to clarify the idea.
\medskip

Consider, for example, an independent and identically distributed  (i.i.d.) sequence $(Z_i)_{i\ge 1}$ of $\mathbb R^d$-valued zero mean random variable with common probability law. A LDP for $S_n=\sum_{i=1}^n Z_i$ will formally say that for $\delta>0$,
$$\mathbb P(|S_n|>n\delta) \approx \exp\{-n\inf\{I(z):|z|\ge \delta\}\},$$
where for $\displaystyle z\in \mathbb R^d, I(z)=\sup_{\alpha\in\mathbb R^d}\big\{\langle\alpha,z\rangle-\log\mathbb E\exp\langle\alpha,Z\rangle\big\}$. Now let $a_n$ be a positive sequence such that $a_n\rightarrow \infty$ and $n^{-1/2}a_n\rightarrow 0$ as $n\rightarrow \infty$. Then the MDP for $S_n $ will say that
$$\mathbb P(|S_n|>n^{1/2}a_n\delta) \approx \exp\{-a^2_n\inf\{I^0(z):|z|\ge \delta\}\},$$
where $I^0(z)=\frac{1}{2}\langle z,{\rm Cov(Z)}^{-1}z \rangle$. Thus the MDP gives estimates on probabilities of deviations of order $n^{1/2}a_n$, which is of lower order than $n$ and with a rate function that is quadratic form.
\medskip

The purpose of this paper is to investigate MDP for both the Hayashi-Yoshida  estimator  as well as for the generalised bipower estimator. These two statistics have been widely studied, both theoretically
and empirically.
\medskip

We consider $X_t=(X_{1,t},X_{2,t})_{t\in [0,T]}$ a 2-dimensional semimartingale, defined on the filtred probability space $(\Omega ,\mathcal F,(\mathcal F_t)_{[0,T]},\mathbb P)$, of the form
\begin{equation}\label{bi-equation}
\left\{\aligned
dX_{1,t}=b_{1} (t,X_t)dt +\sigma_{1,t}dW_{1,t}\\
dX_{2,t}=b_{2}(t,X_t) dt+\sigma_{2,t} dW_{2,t}
\endaligned
\right.
\end{equation}
 where  $W_1=(W_{1,t})_{t\in [0,T]}$ and $W_2=(W_{2,t})_{t\in [0,T]}$ are two correlated Wiener processes
 with $\rho_t={\rm Cov}(W_{1,t},W_{2,t})$, $t\in[0,T]$. Moreover, $\rho_{\cdot}\in [0,1]$
 and $\sigma_{\ell,\cdot}, \ell=1,2$ are both unknown deterministic and
 measurable functions of $t$, $b_{\ell}(\cdot,\cdot), \ell=1,2$ are progressively measurable (possibly unknown) functions. 
\vspace{5pt}

Models of the type (\ref{bi-equation}) and their extensions are widely used in mathematical finance to capture the dynamics of stock prices or interest rates.

\medskip

To provide analysis of MDP in our context, we prove a new result about MDP for $m$-dependent random variables using the Chen-Ledoux type condition. This condition links the speed of the MDP to the queue of the distribution of the random variables.

\medskip

It known that the Chen-Ledoux type condition is a necessary and sufficient condition for the independent and identically distributed random variables, see \cite{EL}. Djellout \cite{DH} using a similar condition has obtained the MDP for martingale difference sequence. Djellout and Guillin \cite{DG} have also obtained the MDP for Markov chain using this conditions. See also the work of Bitseki Penda, Djellout and Proia \cite{BDP} for the MDP of the Durbin Watson statistics.

\medskip

We start with the definition of the LDP.
\medskip

We say that a sequence of random variables $(Z_n)$ with topological state space $\mathcal Z$ satisfies a LDP with speed $\alpha_n$ and good rate function $I(\cdot):\mathcal Z\longrightarrow \mathbb R^+$ if $I$ is lower semi-continuous with compact level set and for every mesurable set $A$, we have
$$-\inf_{x\in \mathring{A}}I(x)\le \liminf_{n\rightarrow \infty}\frac{1}{\alpha_n}\log\mathbb P(Z_n\in A)\le \limsup_{n\rightarrow \infty}\frac{1}{\alpha_n}\log\mathbb P(Z_n\in A)\le -\inf_{x\in \bar{A}}I(x),$$
where  $\mathring{A}$ and $\bar{A}$ denote the interior and the closure of $A$, respectively.
\medskip

Now we will introduce and explain the construction of the two estimators. We start with the Hayashi-Yoshida estimator and we continue next with the generalised bipower estimator.

\medskip

{\bf Hayashi-Yoshida estimator} In this first part, we focus our attention on the estimation of the co-volatility of $X_1$ and $X_2$
$$
\langle X_1, X_2\rangle_T=\int_0^T\sigma_{1,t}\sigma_{2,t}\rho_tdt.
$$

Given the synchronous observations of the processes $\left(X_{1,t_i},X_{2,t_i}\right)_{i=0,...,n}$ a popular statistic to estimate the co-volatility is
$$
{\mathcal C}_n:=\sum_{i=1}^{n}\Delta X_{1}(I^i)\Delta X_{2}(I^i),
$$
where  $I^i=(t_{i-1},t_i]$  and $\Delta X_{\ell}(I^i)=X_{\ell,t_i}-X_{\ell,t_{i-1}}$. This estimator is often called the realized co-volatility estimator, see  \cite{Andersen}.

\vspace{5pt}
The asymptotic distribution of ${\mathcal C}_n$ was formulated by Barndorff et al. \cite{BN1}.  Djellout and Samoura \cite{DS}, Djellout et al. \cite{DGS} obtained the LDP and MDP for the realized covolatility ${\mathcal C}_n$. For more references, one can see \cite{DGW},\cite{HJ},\cite{KO},\cite{MC} and the references therein.
\vspace{5pt}

However, in financial applications, actual transaction data are recorded at irregular times in a nonsynchronous manner, i.e. two transaction prices are usually not observed at the same time.
This fact requires one who adopt ${\mathcal C}_n$ to synchronize the original data a prior, choose a common interval length $h$ first, then impute missing observations by some interpolation scheme such as previous-tick interpolation of linear interpolation. Unfortunately, those procedures
may result in synchronization bias \cite{HY2}.

\vspace{5pt}

Recently, Hayashi and Yoshida \cite{HY2} postulated a new estimator which is free of synchronization
and hence of any bias due to it. To be explicitly,
for the nonsynchronous observations $\left(X_{1,t_i},X_{2,s_j}\right)_{i=0,...,n}^{j=0,...,m}$ with $0=s_0<s_1<...<s_{m}=T, 0=t_0<t_1<...<t_{n}=T, m, n\in \mathbb N$, the Hayashi-Yoshida estimator is defined as
\begin{equation}\label{H-Y}
U_{n,m}:=\sum_{i=1}^{n}\sum_{j=1}^{m}\Delta X_{1}(I^i)\Delta X_{2}(J^j)I_{\{I^i\cap J^j\neq \emptyset\}},
\end{equation}
where $I^i=(t_{i-1},t_i]$, $J^j=(s_{j-1},s_j]$. Under some assumptions on the equation~(\ref{bi-equation}),
Hayashi and Yoshida \cite{HY2},\cite{HN} showed the consistency and asymptotic
normality of $U_{n,m}$ respectively.

\vspace{15pt}
{\bf Generalised Bipower estimator} Recently, the concept of realised bipower variation has built a non-parametric framework for backing out several variational measures of volatility, which has led to a new development in econometrics. Given the observations of the processes $\left(X_{\ell,t_i}\right)_{i=0,...,n}$
with $0=t_0<t_1<...<t_n=T,n\in \mathbb N$, $\ell=1,2$, the realised Bipower variation which is given by
\begin{equation}\label{rbp1}
V_{\ell,1}^{n}(r,q)=\frac{1}{n}\sum_{i=1}^{n-1}\left|\sqrt{n} \Delta X_{\ell}(I^i)\right|^r\left| \sqrt{n}\Delta X_{\ell}(I^{i+1})\right|^q,
\end{equation}
provides a whole variety of estimators for different integrated powers of volatility. For practical convenience, it is standard to take equal spacing, i.e., $t_i-t_{i-1}=\frac{T}{n}:=h$.
An important special case of the class (\ref{rbp1})  is the realised volatility
\begin{equation}\label{rbp}
\sum_{i=1}^{[nt]}|\Delta X_{\ell}(I^i)|^2,
\end{equation}
which is a consistent estimator of the quadratic variation of $X_{\ell}$, i.e.  $\displaystyle \int_0^t\sigma_{\ell,s}^2ds,$
which is often referred to as integrated volatility in the econometric literature.
\vspace{5pt}

In the last years, Nielsen et al. \cite{BN1},\cite{BN2},\cite{BN3} showed the consistency of $V_{\ell,1}^{n}(r,q)$ and introduced the stable CLT for standardised version. Vetter \cite{VM2} extend the results from Jacod \cite{JJ} to the case of bipower variation and he proved the CLT of the bipower variation for continuous semimartingales. In \cite{BN5} they consider the same problems as here where $X$ have jumps and the CLT has been obtained. It was proved that realized bipower variation can estimate integrated power volatility in stochastic volatility models and moreover, under some conditions, it can be a good measure to integrated variance in the presence of jumps. 
\vspace{5pt}

S. Kinnebrock and M. Podolskij \cite{KP} extended the CLT to the bipower variation of general functions called generalised bipower variation:
\begin{equation}\label{lm}
V_{\ell,1}^{n}(g,h)=\frac 1n \sum_{i=1}^{n}g\left(\sqrt n \Delta X_{\ell}(I^i)\right)h\left(\sqrt n\Delta X_{\ell}(I^{i+1})\right),
\end{equation}
where $g$, $h$ are two maps on $\mathbb R$, taking values in $\mathbb R$ and gived some examples from the litterature to which their theory can be applied.

\vspace{5pt}
We know from \cite{BN1} and \cite{BN2} that if $g$ and $h$ are continuously differentiable with $g$, $h$, $g'$ and $h'$ being of at most polynomial growth 
$$V_{\ell,1}^{n}(g,h)\overset{\mathbb P}{\longrightarrow} V_{\ell,1}(g,h):=\int_0^T\Sigma_{\sigma_{\ell,u}}(g)\Sigma_{\sigma_{\ell,u}}(h)du,$$
where 
\begin{equation}\label{Sigma1}\Sigma_{\sigma}(f):=\mathbb E (f(\sigma Z)),\quad Z\sim \mathcal N (0,1),\quad f {\text{ is a real-valued function }},
\end{equation}
where $\overset{\mathbb P}{\longrightarrow} $ denote the convergence in probability.
\vspace{5pt}

In additional, if $h$ and $g$  are  even, we have the CLT
$${\sqrt n}\left(V_{\ell,1}^{n}(g,h)-V_{\ell,1}(g,h)\right)\rightarrow {\mathcal N}(0,\Sigma_{\ell}(g,h))$$
where 
\begin{equation}\label{Sigma2}\Sigma_{\ell}(g,h):=\int_0^T\left(\Sigma_{\sigma_{\ell,s}}(g^2)\Sigma_{\sigma_{\ell,s}}(h^2)+2\Sigma_{\sigma_{\ell,s}}(g)\Sigma_{\sigma_{\ell,s}}(h)\Sigma_{\sigma_{\ell,s}}(gh)-3\Sigma^2_{\sigma_{\ell,s}}(g)\Sigma^2_{\sigma_{\ell,s}}(h)\right)ds.
\end{equation}

Hence, the establishment of the MDPfor  the previously statistics is the natural continuation following the proof of central limit theorems and the law of large numbers.

\vspace{15pt}
This article is structured as follows. In Section 2 we present the main theoretical results and  we state the proofs in the Section 3. The appendix is devoted to state and prove the main result about the MDP for $m$-dependent random variables.


\section{Main Results}

\subsection{Moderate deviations for the  Hayashi-Yoshida estimator }

Firstly, a reduced design with respect to $(I^i)_{i=1,...,n}$ will be constructed in the following manner.
We collect all $I^i$s such that $I^i\subset J^j$ and combine them into a new interval; if such  $I^i$ does not
exist, do nothing. Then collecting all such
intervals and re-labeling them from left to right yields a new partition of $(0,T]$,
denoted $\left(\hat{I}^{i}\right)_{i=1,...,\hat{n}}$. Due to the bilinearity of $U_{n,m}$ given in (\ref{H-Y}),
\begin{equation}\label{Un-expression}
U_{n,m}=\sum_{i=1}^{\hat{n}}\sum_{j=1}^{m}\Delta X_{1}(\hat{I}^{i})\Delta X_{2}(J^j)K^{\hat{I}}_{ij}
=\sum_{{i}=1}^{\hat{n}}\Delta X_{1}(\hat{I}^{{i}})\Delta X_{2}\big(\underset{j\in \hat{J}({i})}{\cup}  J^j\big),
\end{equation}
where~$\hat{J}({i}):=\left\{1\leq j\leq m: K^{\hat{I}}_{ij}\neq0\right\}$
 with $K^{\hat{I}}_{ij}=I_{\{\hat{I}^i\cap J^j\neq \emptyset\}}$. Then, we can find that each $J^j$
contains at most one $\hat{I}^{i}$, which implies that the random variable sequence
$$
\left\{\Delta X_{1}(\hat{I}^{i})\Delta X_{2}\big(\underset{j\in \hat{J}({i})}{\cup}  J^j\big)\right\}_{1\leq i\leq\hat{n}}
$$
are $2$-dependent.
Now, the number $\hat{n}$ can be formulated as follows. Define
$$
\tau_1=\inf\{1\leq i\leq n: I^i\nsubseteq J^1\},~ \varsigma_1=\sup\{1\leq j\leq m: I^{\tau_1}\cap J^j\neq\emptyset\}
$$
and
$$
\tau_k=\inf\{\tau_{k-1}< i\leq n: I^i\nsubseteq J^{\varsigma_{k-1}}\},~ \varsigma_k=\sup\{\varsigma_{k-1}<j\leq m: I^{\tau_k}\cap J^j\neq\emptyset\},
$$
with~$\tau_0=0,~\varsigma_0=1$ and~$\inf{\emptyset}=+\infty,~\sup{\emptyset}=0$. Let~$n_0=\sup\{k: \tau_k<+\infty\}$.
Then, one can conclude that
\begin{equation}\label{number}
n_0\leq\hat{n}\leq n_0+\sum_{k=1}^{n_0}I_{\left\{\tau_k-\tau_{k-1}>1\right\}}+I_{\left\{\tau_{n_0}<n\right\}}
\leq 2n_0+1.
\end{equation}

Let $A_{\ell,t}=\int_0^tb_{\ell}(s,\omega)ds$, $t\in[0,T]$ and $\ell=1,2$. When the drift
$b_{\ell}(t,\omega)$ is known, we can consider the following estimator
\begin{equation}\label{def-V}
\begin{aligned}
V_{n,m}:&=\sum_{i=1}^{n}\sum_{j=1}^{m}\Delta X_{1}^0(I^i)\Delta X_{2}^0(J^j)I_{\{I^i\cap J^j\neq \emptyset\}}\\
&=\sum_{{i}=1}^{\hat{n}}\Delta X_{1}^0(\hat{I}^{{i}})\Delta X_{2}^0\big(\underset{j\in \hat{J}({i})}{\cup}  J^j\big).
\end{aligned}
\end{equation}
where $X_{\ell,t}^0 = X_{\ell,t} - A_{\ell,t}$.\par

 For any Borel set $I\subset[0,T]$, define
$$
\nu(I):=\int_I\sigma_{1,t}\sigma_{2,t}\rho_tdt,\qquad \nu_{\ell}(I):=\int_I\sigma_{\ell,t}^2dt,~\ell=1,2.
$$
Let~$r_{n,m}:=\max_{1\leq i\leq n}|I^i|\vee\max_{1\leq j\leq m}|J^j|$, the largest interval size. We always
assume that as $n, m\to+\infty$, $r_{n,m}\to0$, which implies that $n_0\to+\infty$.
 For simplicity, let $m=m(n)$ and as $n\to+\infty$,  $m\to+\infty$. Moreover,
write
$$
U_{n,m}=U_n,\quad V_{n,m}=V_n,\quad r_{n,m}=r_n.
$$
To establish the MDP for $V_n$, we introduce the following conditions.
\vspace{5pt}

\noindent \textbf{(C1)} There exist a sequence of positive numbers $(c_n)_{n\geq1}\subset(0,1)$ and some constant
$\Sigma\in(0,+\infty)$ such that, as $n\longrightarrow+\infty$, $c_n\longrightarrow 0$ and
$$
c_n^{-1}\left(\sum_{i=1}^{\hat n}\sum_{j=1}^{m}\nu_1(\hat{I}^i)\nu_2(J^j)K^{\hat{I}}_{ij}+\sum_{i=1}^{\hat n}\nu^2(I^i)
+\sum_{j=1}^{m}\nu^2(J^j)-\sum_{i=1}^{\hat n}\sum_{j=1}^{m}\nu^2(I^i\cap J^j)\right)\longrightarrow \Sigma.
$$

\noindent \textbf{(C2)} Let $(b_n)_{n\geq1}$ a sequence of positive numbers 
 such that, as $n\longrightarrow+\infty$,
 $$
 b_n\longrightarrow+\infty,\quad b_n\sqrt{c_n}\longrightarrow 0\quad{\rm and}\quad  \frac{b_n}{\sqrt{c_n}\log n}\longrightarrow \infty.$$
\begin{thm}\label{mdp-1}
Under conditions {\textbf{(C1)}} and {\textbf{(C2)}}, the sequence
\begin{equation}\label{esti}
\left\{\frac{1}{b_n\sqrt{c_n}}\left(V_n-\int_{0}^{T}\sigma_{1,t}\sigma_{2,t}\rho_tdt\right)\right\}_{n\geq1}
\end{equation}
satisfies the LDP with speed $b^2_n$ and rate
function  $\displaystyle L(x)=\frac{x^2}{2\Sigma}$.\quad$\Box$
\end{thm}

\begin{rem} The condition  \textbf{(C1)} is similar to the \textbf{Condition (C2)} in the paper of Hayashi-Yoisda \cite{HY2}. It postulates the asymptotic connection between the observation and the variance-covariance structure of the given processes. The constant $\Sigma$ serves as the asymptotic variance of the proposed estimator. The rescaling factor $c_n^{-1}$ may be interpreted as the 'average number' off the observation times.
\end{rem}

\begin{rem} The condition \textbf{(C2)} is weaker to the condition (1.10) in Theorem 1.3 of the paper of Djellout et al. \cite{DGW}. Indeed in the last paper,  we need additionnel condition which relies the sequence of the MDP $(b_n)_{n\ge 1}$ to the integrated covolatility. This condition can be expressed in the following way in the actual context
\begin{equation}\label{tay}
b_nc_{n}^{-1/2}\alpha_n\longrightarrow 0,
 \end{equation}
 where $\alpha_n=\max_{\ell=1,2}\max_{1\leq i\leq n}\max_{1\leq j\leq m}\left(\nu_{\ell}(I^i)\vee\nu_{\ell}(J^j)\right)$. This condition (\ref{tay})  comes from the method of the proof based on the Taylor expansion of the Laplace transform of the estimator (\ref{esti}).

\end{rem}

\begin{rem}\label{rmk-mdp1}
By the construction of partition $(\hat{I}^i)_{i=1,...,\hat{n}}$, we obtain immediately that
\begin{align*}
&\sum_{i=1}^{\hat n}\sum_{j=1}^{m}\nu_1(\hat{I}^i)\nu_2(J^j)K^{\hat{I}}_{ij}+\sum_{i=1}^{\hat n}\nu^2(\hat{I}^i)
+\sum_{j=1}^{m}\nu^2(J^j)-\sum_{i=1}^{\hat n}\sum_{j=1}^{m}\nu^2(\hat{I}^i\cap J^j)\\
&=\sum_{i=1}^{n}\sum_{j=1}^{m}\nu_1(I^i)\nu_2(J^j)K^{{I}}_{ij}+\sum_{i=1}^{n}\nu^2(I^i)
+\sum_{j=1}^{m}\nu^2(J^j)-\sum_{i=1}^{n}\sum_{j=1}^{m}\nu^2(I^i\cap J^j).
\end{align*}
\end{rem}

\begin{rem}
Similar to $(\hat{I}^i)_{i=1,...,\hat{n}}$, a reduced design with respect to $(J^i)_{i=1,...,m}$
can  also be constructed as follows.
We collect all $J^j$s such that $J^j\subset I^i$ and combine them into a new interval; if such  $J^j$ does not exist, do nothing. Then collecting all such intervals and re-labeling them from left to right yields a new partition of $(0,T]$, denoted $(\hat{J}^{j})_{j=1,...,\tilde{n}}$. Then $V_n$ has the formula,
\begin{equation}\label{Un-expression-0}
V_n=
\sum_{j=1}^{\tilde{n}}\Delta X_{1}^0\big(\underset{i\in \hat{I}(j)}{\cup}I^i\big)\Delta X_{2}^0(\hat{J}^{j}),
\end{equation}
where $\hat{I}(j):=\left\{1\leq i\leq n: K^{\hat{J}}_{ij}\neq\emptyset\right\}$
 with $K^{\hat{J}}_{ij}=I_{\{I^i\cap \hat{J}^j\neq 0\}}$. Letting $\tilde{n}_0=\sup\{k: \varsigma_k>0\}$, then $\tilde{n}_0=n_0-1$ and
$$
n_0\leq\tilde{n}\leq n_0+\sum_{k=1}^{n_0}I_{\left\{\varsigma_k-\varsigma_{k-1}>1\right\}}+1\leq 2n_0+1.
$$
Therefore, we can conclude that the MDP of $V_n$ are independent of the different
partitions $(\bar{I}^i)_{i=1,...,\bar{n}}$ and $(\hat{J}^{j})_{j=1,...,\tilde{n}}$ in (\ref{Un-expression}) and
 (\ref{Un-expression-0}).
 \end{rem}
 \vspace{5pt}

Now we turn to the MDP for $U_n$. We need the following two additional conditions.
\vspace{5pt}

\noindent \textbf{(C3) } For $\ell=1,2$, $b_{\ell}(\cdot,\cdot)\in L^{\infty}\left({\rm d}t\otimes \mathbb {P}\right),$

\noindent \textbf{(C4) } As $n\longrightarrow +\infty$,  $\displaystyle r_n^2/c_n\longrightarrow 0$.
\vspace{5pt}

\begin{thm}\label{mdp-2}
Under conditions {\textbf{(C1)}} through {\textbf{(C4)}}, the sequence
$$
\left\{\frac{1}{b_n\sqrt{c_n}}\left(U_n-\int_{0}^{T}\sigma_{1,t}\sigma_{2,t}\rho_tdt\right)\right\}_{n\geq1}
$$
satisfies the LDP with speed $b^2_n$ and rate function $\displaystyle L(x)=\frac{x^2}{2\Sigma}$. \quad$\Box$
\end{thm}

\begin{rem} Suppose synchronous and equidistant sampling, i.e. $I^i\equiv J^i$ and $|I^i|\equiv T/n$.
Then, we can see that
$$
n_0=n,\quad c_n=1/n,\quad \Sigma=\int_0^T\sigma_{1,t}^2\sigma_{2,t}^2\left(1+\rho_t^2\right)dt.
$$
Moreover, if $\sigma_{\ell,\cdot}\in L^{\infty}(dt)$ for $\ell=1,2$, condition {\textbf{(C2)}} equivalent to
\begin{equation}\label{rmk1-eq1}
b_n\longrightarrow+\infty,\quad b_nn^{-1/2}\longrightarrow 0,\quad{\rm as} \quad n\longrightarrow+\infty.
\end{equation}
Under (\ref{rmk1-eq1}), from our Theorem~\ref{mdp-1}, it follows that the sequence
\begin{equation}\label{suite}
\left\{\frac{\sqrt n}{b_n}\left(V_n-\int_{0}^{T}\sigma_{1,t}\sigma_{2,t}\rho_tdt\right)\right\}_{n\geq1}
\end{equation}
satisfies the LDP with speed $b^2_n$ and rate
function
$$
L_0(x)=\frac{x^2}{2\int_0^T\sigma_{1,t}^2\sigma_{2,t}^2\left(1+\rho_t^2\right)dt}.
$$
This result can also be obtained by Theorem 2.5 in Djellout et al. \cite{DGS}.
\vspace{5pt}

In this case, we have $r_n=1/n$, and condition {\textbf{(C4)}}
holds automatically.
\end{rem}

\begin{rem}
 We consider a nonsynchrnous alternating sampling at odd/even times. We now consider the following deterministic, regularly spaced sampling scheme. The diffusion $X_1$ is sampled at 'odd' times, i.e., $t=\frac{2k-1}{n} T, k=1,2,\cdots,n$ while $X_2$ is at 'even' times, $t=\frac{2k}{2n}T$. Hence $X_1$ and $X_2$ are sampled in a nonsynchronous, alternating way.
So we have that the sequence given in (\ref{suite})  satisfies the LDP with speed $b_n^2$ and rate function $\displaystyle L(x)=\frac{x^2}{2\Sigma}$ with 
$$\Sigma=\int_0^T(\sigma_{1,t}\sigma_{2,t})^2(2+\frac{3}{2}\rho_t^2)dt,$$
see \cite{HY3} for the identification of $\Sigma$.
\end{rem}

\subsection{Moderate deviations for bipower variation of general function}

\begin{prop} \label{mdp_bv1}
Let $(b_n)$ be a sequence of positive numbers such that 
\begin{equation} \label{bv_bn1} b_n\longrightarrow \infty\quad{\rm  and}\quad b_n/\sqrt{n}\longrightarrow 0.\end{equation}
Assume that $g$ and $h$ are two evens, continuous differentiable with at most polynomial growth such that
\begin{equation} \label{eq_bv1}
\begin{aligned}
&\limsup_{n\rightarrow \infty}\frac{1}{b_n^2}\log n\max_{j=1}^n\mathbb P\biggl(\biggl|g\left(\sqrt n \Delta X_{\ell}(I^{j})\right)h\left(\sqrt n\Delta X_{\ell}(I^{j+1})\right)\\
&\hspace*{2.2cm}-\mathbb{E}\left(g(\sqrt n \Delta X_{\ell}(I^{i}))\right)\mathbb{E}\left(h(\sqrt n \Delta X_{\ell}(I^{i+1}))\right)\biggr|> b_n\sqrt{n}\biggr)=-\infty.
\end{aligned}
\end{equation}
So, we have that the sequence
\begin{equation}\label{esti2}
\left\{\frac{\sqrt n}{b_n}\left(V_{\ell,1}^{n}(g,h)-V_{\ell,1}(g,h)\right)\right\}_{n\ge 1}\end{equation}
satisfies the LDP with speed $b_n^2$ and rate function $I_{g,h}^{\ell}(x)=x^2/2\Sigma_{\ell}(g,h),$
where $\Sigma_{\ell}(g,h)$ is given in (\ref{Sigma2}).
\end{prop}

\begin{rem} The condition (\ref{bv_bn1}) is weaker to the condition (1.10) in Theorem 1.3 of the paper of Djellout et al. \cite{DGW}. Indeed in the last paper,  we need additionnel condition which relies the sequence of the MDP $(b_n)_{n\ge 1}$ to the integrated volatility. This condition can be expressed in the following way in the actual context
\begin{equation}\label{tay2}
\beta_nb_n\sqrt n\longrightarrow 0,
\end{equation}
where $\beta_n=\max_{1\le i\le n}\int_{I^i}\Sigma_{\sigma_{\ell,s}}(g)\Sigma_{\sigma_{\ell,s}}(h)ds$ with $\Sigma_{\sigma}(f)$ given in (\ref{Sigma1}).
This condition (\ref{tay2})  comes from the method of the proof based on the Taylor expansion of the Laplace transform of the estimator (\ref{esti2}).
\end{rem}

\begin{rem}
Realised bipower variation, which is probably the most important subclass of our model, corresponds to the functions $g(x)=|x|^r$ and $h(x)=|x|^{q}$. In this case and under the condition that

\begin{equation}\label{rl}
\limsup_{n\rightarrow \infty}\frac{1}{b_n^2}\log n\max_{j =1}^n\mathbb P\left( |\sqrt n\Delta X_{\ell}(I^{j})|^{r}|\sqrt n\Delta X_{\ell}(I^{j+1})|^{q}>b_n\sqrt n\right)=-\infty.
\end{equation}

We have that the sequence
$$\frac{\sqrt n}{b_n}\left(V^n_{\ell,1}(r,q)-\mu_r\mu_{q}\int_0^T|\sigma_{\ell,s}|^{r+q}ds\right),$$
 satisfies the MDP with speed $b_n^2$ and rate function $I_{g,h}^{\ell}(x)=x^2/2\Sigma_{\ell}(g,h),$
with
$$\Sigma_{\ell}(g,h):=\Sigma_{\ell}(r,q)=(\mu_{2r}\mu_{2q}+2\mu_r\mu_q\mu_{r+q}-3\mu_r^2\mu_{q}^2)\int_0^T|\sigma_{\ell,u}|^{2(r+q)}du,$$
where $\mu_r=\mathbb E|z|^r$ with $z\sim \mathcal N(0,1)$.
\vspace{10pt}

If $r=2$ and $q=0$ we are in the quadratic case and the condition (\ref{rl}) is satisfied for all the sequence $(b_n)$. The result was already obtained in Djellout-Guillin-Wu \cite{DGW}.
\vspace{10pt}

If $r+q<2$ the condition (\ref{rl}) is satisfied for every sequence $b_n$.

\end{rem}

\begin{rem} The cubic power of returns is a special case of the generalised bipower variation, corresponds to the functions $g(x)=|x|^3$ and $h(x)=1$. Then, under the condition
$$
\limsup_{n\rightarrow \infty}\frac{1}{b_n^2}\log n\max_{j=1}^n\mathbb P\left( |\sqrt n\Delta X_{\ell}(I^{j})|^{3}>b_n\sqrt n\right)=-\infty,
$$
the sequence $\displaystyle\frac{\sqrt n}{b_n}\left(V^n_{\ell,1}(|x|^3,1)-\mu_3\int_0^T|\sigma_{\ell,s}|^3ds\right)$ satisfies the LDP with speed $b_n^2$ and rate function $\displaystyle I (u)=24x^2/\int_0^T \sigma_s^{6}ds.$
\end{rem}


\section{proofs}

In this section, we will use the appendix's Proposition~\ref{mdpmdep} about the the MDP for $m$-dependent random variables sequences which are not necessarily stationnary under the Chen-Ledoux type condition.

\vspace{10pt}

\noindent\emph{\textbf{Proof of Theorem \ref{mdp-1}}}

We start this section by calculating the variance of $V_n$.
Different from the method used in Hayashi and Yoshida \cite{HN},
our approach relies on (\ref{Un-expression}), the $2$-dependent representation of $U_n$. Moreover, the analysis of MDP will benefit from this straightforward thoughts.

\begin{lem}\label{lem-variance}
For the estimator $V_n$ defined by (\ref{def-V}), we have
$$
{\rm Var}(V_n)=\sum_{i=1}^{\hat{n}}\sum_{j=1}^{m}\nu_1(\hat{I}^{i})\nu_2(J^j)K^{\hat{I}}_{ij}+\sum_{i=1}^{\hat{n}}\nu^2(\hat{I}^{i})
+\sum_{j=1}^{m}\nu^2(J^j)
-\sum_{i=1}^{\hat{n}}\sum_{j=1}^{m}\nu^2\left(\hat{I}^{i}\cap J^{j}\right).
$$
\end{lem}

\begin{proof}
For $\ell=1,2$, letting $\displaystyle
M_{\ell,t}=\int_0^t\sigma_{\ell,s}dW_{\ell,s},\quad t\in[0,T],$
then
$$
V_n=\sum_{{i}=1}^{\hat{n}}\Delta M_{1}(\hat{I}^{{i}})\Delta M_{2}\big(\underset{j\in \hat{J}({i})}{\cup} J^j\big).
$$

Since $\displaystyle \left\{\Delta M_{1}(\hat{I}^{i})\Delta M_{2}\big(\underset{ j\in \hat{J}(i) }{\cup} J^j\big)\right\}_{1\leq i\leq\hat{n}}$
are $2$-dependent, one can write that
\begin{align*}
{\rm Var}(V_n)&=\sum_{\hat{i}=1}^{\hat{n}}{\rm Var}\left(\Delta M_{1}(\hat{I}^{i})\Delta M_{2}\big(\underset{j\in \hat{J}(i)}{\cup}J^j\big)\right)\\
&\quad+\sum_{i=1}^{\hat{n}}{\rm Cov}\left(\Delta M_{1}(\hat{I}^{i})\Delta M_{2}\big(\underset{j\in \hat{J}(i)}{\cup}J^j\big),
\Delta M_{1}(\hat{I}^{i+1})\Delta M_{2}\big(\underset{j\in \hat{J}(i+1)}{\cup}J^j\big)\right)\\
&\quad+\sum_{i=1}^{\hat{n}}{\rm Cov}\left(\Delta M_{1}(\hat{I}^{i})\Delta M_{2}\big(\underset{j\in \hat{J}(i)}{\cup}J^j\big),
\Delta M_{1}(\hat{I}^{i+2})\Delta M_{2}\big(\underset{j\in \hat{J}(i+2)}{\cup}J^j\big)\right)\\
&\quad+\sum_{i=1}^{\hat{n}}{\rm Cov}\left(\Delta M_{1}(\hat{I}^{i})\Delta M_{2}\big(\underset{j\in \hat{J}(i)}{\cup}J^j\big),
\Delta M_{1}(\hat{I}^{i-1})\Delta M_{2}\big(\underset{j\in \hat{J}(i-1)}{\cup}J^j\big)\right)\\
&\quad+\sum_{i=1}^{\hat{n}}{\rm Cov}\left(\Delta M_{1}(\hat{I}^{i})\Delta M_{2}\big(\underset{j\in \hat{J}(i)}{\cup}J^j\big),
\Delta M_{1}(\hat{I}^{i-2})\Delta M_{2}\big(\underset{{j\in \hat{J}(i-2)}}{\cup}J^j\big)\right)\\
&:=D_1+D_2+D_3+D_4+D_5,
\end{align*}
where~$\hat{I}^{i}=\emptyset$ for $i<1$ or $i>\hat{n}$.

We will only consider $D_1$ and $D_2$, the other terms can be dealt with in the same way.
A simple calculation gives us
\begin{equation}\label{D1}
\begin{aligned}
D_1&=\sum_{i=1}^{\hat{n}}\left(\nu_1(\hat{I}^{i})\nu_2\big(\underset{j\in \hat{J}(i)}{\cup}J^j\big)+\nu^2(\hat{I}^{i})\right)\\
&=\sum_{i=1}^{\hat{n}}\sum_{j=1}^{m}\nu_1(\hat{I}^{i})\nu_2(J^j)K^{\hat{I}}_{ij}+\sum_{\hat{i}=1}^{\hat{n}}\nu^2(\hat{I}^{i}).
\end{aligned}
\end{equation}
Moreover, from the fact
\begin{align*}
\Delta M_{2}\big(\underset{j\in \hat{J}(i)}{\cup}J^j\big)&=\Delta M_{2}\big(\hat{I}^{i}\big)
+\Delta M_{2}\big(\underset{j\in \hat{J}(i)}{\cup}J^j\cap \hat{I}^{i+1}\big)+\Delta M_{2}\big(\underset{j\in \hat{J}(i)}{\cup}J^j\cap \hat{I}^{i+2}\big)\\
&\quad+\Delta M_{2}\big(\underset{j\in \hat{J}(i)}{\cup}J^j\cap \hat{I}^{i-1}\big)
+\Delta M_{2}\big(\underset{j\in \hat{J}(i)}{\cup}J^j\cap \hat{I}^{i-2}\big),
\end{align*}
it follows that
\begin{align*}
&\mathbb E\left(\Delta M_{1}(\hat{I}^{i})\Delta M_{1}(\hat{I}^{i+1})
\Delta M_{2}\big(\underset{j\in \hat{J}(i)}{\cup}J^j\big)\Delta M_{2}\big(\underset{j\in \hat{J}(i+1)}{\cup}J^j\big)\right)\\
&=\nu\big(\hat{I}^{i}\big)\nu\big(\hat{I}^{i+1}\big)
+\nu\big(\underset{j\in \hat{J}(i)}{\cup}J^j\cap \hat{I}^{i+1}\big)
\nu\big(\underset{j\in \hat{J}(i+1)}{\cup}J^j\cap \hat{I}^{i}\big)\\
&=\nu\big(\hat{I}^{i}\big)\nu\big(\hat{I}^{i+1}\big)
+\sum_{j_1=1}^{m}\sum_{j_2=1}^{m}\nu\big(J^{j_1}\cap \hat{I}^{i+1}\big)
\nu\big(J^{j_2}\cap \hat{I}^{i}\big)K^{\hat{I}}_{ij_1}K^{\hat{I}}_{(i+1)j_2}.
\end{align*}
We can write the second term in the above equality as
$$
\nu\big(J^{j_1}\cap \hat{I}^{i+1}\big)
\nu\big(J^{j_2}\cap \hat{I}^{i}\big)K^{\hat{I}}_{ij_1}K^{\hat{I}}_{(i+1)j_2}=\nu\big(J^{j_1}\cap \hat{I}^{i+1}\big)
\nu\big(J^{j_2}\cap \hat{I}^{i}\big)K^{\hat{I}}_{ij_1}K^{\hat{I}}_{(i+1)j_2}K^{\hat{I}}_{(i+1)j_1}K^{\hat{I}}_{ij_2}.
$$
By the following fact,
$$
K^{\hat{I}}_{ij_1}K^{\hat{I}}_{(i+1)j_2}K^{\hat{I}}_{(i+1)j_1}K^{\hat{I}}_{ij_2}\neq0\Leftrightarrow j_1=j_2,
$$
we can obtain immediately that
\begin{align*}
&\mathbb E\left(\Delta M_{1}(\hat{I}^{i})\Delta M_{1}(\hat{I}^{i+1})
\Delta M_{2}\big(\underset{j\in \hat{J}(i)}{\cup}J^j\big)\Delta M_{2}\big(\underset{j\in \hat{J}(i+1)}{\cup}J^j\big)\right)\\
&=\nu\left(\hat{I}^{i}\right)\nu\left(\hat{I}^{i+1}\right)
+\sum_{j=1}^{m} \nu\left(\hat{I}^{i+1}\cap J^{j}\right)
\nu\left(\hat{I}^{i}\cap J^{j}\right),
\end{align*}
which implies that
\begin{equation}\label{D2}
D_2=\sum_{i=1}^{\hat{n}}\sum_{j=1}^{m} \nu\big(\hat{I}^{i+1}\cap J^{j}\big)
\nu\big(\hat{I}^{i}\cap J^{j}\big).
\end{equation}

Applying the same arguments to $D_3, D_4, D_5$,
\begin{equation}\label{D3}
D_3=\sum_{i=1}^{\hat{n}}\sum_{j=1}^{m} \nu\big(\hat{I}^{i+2}\cap J^{j}\big)
\nu\big(\hat{I}^{i}\cap J^{j}\big),
\end{equation}
\begin{equation}\label{D4}
D_4=\sum_{i=1}^{\hat{n}}\sum_{j=1}^{m} \nu\big(\hat{I}^{i-1}\cap J^{j}\big)
\nu\big(\hat{I}^{i}\cap J^{j}\big),
\end{equation}
\begin{equation}\label{D5}
D_5=\sum_{i=1}^{\hat{n}}\sum_{j=1}^{m} \nu\big(\hat{I}^{i-2}\cap J^{j}\big)
\nu\big(\hat{I}^{i}\cap J^{j}\big).
\end{equation}

Putting~(\ref{D2}) through ~(\ref{D5}) together, and noting fact that if $K^{\hat{I}}_{ij}\neq0$, then $$
J^j=\cup_{\ell=-2}^{2}\left(I^{i+\ell}\cap J^j\right),
$$
we have
\begin{align*}
D_2+D_3+D_4+D_5&=
\sum_{i=1}^{\hat{n}}\sum_{j=1}^{m} \left(\nu(J^j)-\nu\big(\hat{I}^{i}\cap J^{j}\big)\right)
\nu\big(\hat{I}^{i}\cap J^{j}\big)\\
&=
\sum_{j=1}^{m}\nu^2(J^j)
-\sum_{i=1}^{\hat{n}}\sum_{j=1}^{m}\nu^2\big(\hat{I}^{i}\cap J^{j}\big).
\end{align*}
Together with (\ref{D1}), we can conclude
$$
{\rm Var}(V_n)=\sum_{i=1}^{\hat{n}}\sum_{j=1}^{m}\nu_1(\hat{I}^{i})\nu_2(J^j)K^{\hat{I}}_{ij}+\sum_{i=1}^{\hat{n}}\nu^2(\hat{I}^{i})
+\sum_{j=1}^{m}\nu^2(J^j)
-\sum_{i=1}^{\hat{n}}\sum_{j=1}^{m}\nu^2\left(\hat{I}^{i}\cap J^{j}\right).
$$
\end{proof}

\begin{rem}\label{rmk-1}
From the proof of Lemma \ref{lem-variance}, each $\Delta M_{1}(\hat{I}^{i})\Delta M_{2}\big(\underset{j\in \hat{J}(i)}{\cup}J^j\big),$ for $1\leq i\leq\hat{n}$
contributes to the variance of $V_n$ with
$$
\sum_{j=1}^{m}\nu_1(\hat{I}^{i})\nu_2(J^j)K^{\hat{I}}_{ij}+\nu^2(\hat{I}^{i})
+\sum_{j=1}^{m} \left(\nu(J^j)-\nu\big(\hat{I}^{i}\cap J^{j}\big)\right)
\nu\big(\hat{I}^{i}\cap J^{j}\big).\quad\Box\\
$$
\end{rem}


\noindent\emph{\textbf{Proof of Theorem \ref{mdp-1}}}
For the expectation of $V_n$, we have
$$
\mathbb EV_n=\sum_{i=1}^{\hat{n}}\nu\big(\hat{I}^{i}\big)
=\int_{0}^{T}\sigma_{1,t}\sigma_{2,t}\rho_tdt.
$$
Therefore,
\begin{equation}\label{Un-expression-1}
\begin{aligned}
V_n-\int_{0}^{T}\sigma_{1,t}\sigma_{2,t}\rho_tdt
&=\sum_{i=1}^{\hat{n}}\left(\Delta M_{1}(\hat{I}^{i})\Delta M_{2}\big(\underset{j\in \hat{J}(i)}{\cup}J^j\big)
-\nu\big(\hat{I}^{i}\big)\right):=\sum_{i=1}^{\hat{n}}Y_{i},
\end{aligned}
\end{equation}
where $Y_{i}=\Delta M_{1}(\hat{I}^{i})\Delta M_{2}\big(\underset{j\in \hat{J}(i)}{\cup}J^j\big)-\nu\big(\hat{I}^{i}\big)$. Consequently, the random variable sequence~$\{Y_{i}\}_{1\leq i\leq\hat{n}}$ are $2$-dependent.  Since $Y_i$ is a product of two gaussian random variable, so there exist $\delta>0$ such that 
$$\lim_{n\rightarrow \infty}\sup_{1\le i\le n} \mathbb E(e^{\delta | Y_i|})<\infty.$$

Applying the exponential Chebyshev inequality yields with $\delta$ as before

\begin{align*}
\lim_{n\rightarrow \infty}\frac{1}{b_n^2}\log\left[n\max_{i=1}^n\mathbb P(|Y_i|\ge \frac{b_n}{\sqrt {c_n}})\right]&\le \lim_{n\rightarrow \infty}\left(\frac{\log(n)}{b_n^2}-\delta\frac{1}{b_n\sqrt{c_n}}+\frac{1}{b_n^2}\log\max_{i=1}^n\mathbb E(e^{\delta | Y_i|})\right)\\
&= \lim_{n\rightarrow \infty}\frac{1}{b_n\sqrt{c_n}}\left(\frac{\log(n)\sqrt{c_n}}{b_n}-\delta\right)\\
&=\lim_{n\rightarrow \infty}-\frac{1}{b_n\sqrt{c_n}}\delta=-\infty.
\end{align*}

Hence, with appendix Proposition~\ref{mdpmdep}, the Theorem~\ref{mdp-1} holds.
\quad$\Box$

\vspace{10pt}

\noindent\emph{\textbf{Proof of Theorem \ref{mdp-2}}}  From Theorem \ref{mdp-1}, we only need to show the exponential equivalence of
$U_n$ and $V_n$, i.e. for any $\delta>0$
\begin{equation}\label{Un-Vn}
\lim_{n\to+\infty}\frac{1}{b_n^2}\log \mathbb P\left(\frac{1}{b_nc_n^{1/2}}|U_n-V_n|\geq\delta\right)=-\infty.
\end{equation}

In fact, by simple calculation, we can write that
\begin{align*}
|U_n-V_n|&=\sum_{i=1}^{\hat{n}}\left|\Delta A_1(\hat{I}_i)\Delta A_2(\underset{j\in \hat{J}({i})}{\cup}J^j)\right|
+\sum_{i=1}^{\hat{n}}\left|\Delta A_1(\hat{I}_i)\Delta M_2(\underset{j\in \hat{J}({i})}{\cup}J^j)\right|\\
&\qquad\quad+\sum_{i=1}^{\hat{n}}\left|\Delta M_1(\hat{I}_i)\Delta A_2(\underset{j\in \hat{J}({i})}{\cup}J^j)\right|\\
&\leq3Tr_n\max_{\ell=1,2}\|b_{\ell}\|_{\infty}+r_n\|b_{1}\|_{\infty}\sum_{i=1}^{\hat{n}}\left|\Delta M_2(\underset{j\in \hat{J}({i})}{\cup}J^j)\right|+3r_n\|b_{2}\|_{\infty}\sum_{i=1}^{\hat{n}}\left|\Delta M_1(\hat{I}_i)\right|\\
&\leq3Tr_n\max_{\ell=1,2}\|b_{\ell}\|_{\infty}+3r_n\|b_{1}\|_{\infty}\sum_{i=1}^{\hat{n}}\left|\Delta M_2(\hat{I}_i)\right|+3r_n\|b_{2}\|_{\infty}\sum_{i=1}^{\hat{n}}\left|\Delta M_1(\hat{I}_i)\right|.\\
\end{align*}

In the calculation above, we have used Lemma 5 in \cite{HY2}, which says that 
$$\sum_{i=1}^{\hat{n}}|\underset{j\in \hat{J}({i})}{\cup}J^j|\le 3T.$$

Using condition {\bf(C4)}, we have as $n\longrightarrow\infty$
$$
\frac{3Tr_n\max_{\ell=1,2}\|b_{\ell}\|_{\infty}}{b_nc_n^{1/2}}\longrightarrow 0.
$$

Therefore, to prove (\ref{Un-Vn}), we only need to show that for any $\delta>0$
\begin{equation}\label{eq-1}
\lim_{n\to+\infty}\frac{1}{b_n^2}\log \mathbb P\left(\frac{r_n}{b_nc_n^{1/2}}\sum_{i=1}^{\hat{n}}\left|\Delta M_1(\hat{I}_i)\right|\geq\delta\right)=-\infty,
\end{equation}
and
\begin{equation}\label{eq-2}
\lim_{n\to+\infty}\frac{1}{b_n^2}\log\mathbb P\left(\frac{r_n}{b_nc_n^{1/2}}\sum_{i=1}^{\hat{n}}\left|\Delta M_2(\hat{I}_i)\right|\geq\delta\right)=-\infty.
\end{equation}

In fact, according to Chebyshev's inequality, for (\ref{eq-1}), 
\begin{align*}
\mathbb P\left(\frac{r_n}{b_nc_n^{1/2}}\sum_{i=1}^{\hat{n}}\left|\Delta M_1(\hat{I}_i)\right|\geq\delta\right)&\leq\inf_{\lambda>0}\exp\left\{-\lambda\delta+\frac{\lambda^2r_n^2\int_0^T\sigma_{1,t}^2dt}{2b_n^2c_n}\right\}\\
&\leq\exp\left\{-\frac{b_n^2c_n\delta^2}{2r_n^2\int_0^T\sigma_{1,t}^2dt}\right\},
\end{align*}
which implies that (\ref{eq-1}) by condition {\bf(C4)}. Moreover, (\ref{eq-2}) can be established similarly.
\quad$\Box$

\vspace{10pt}

\noindent\emph{\textbf{Proof of Proposition~\ref{mdp_bv1}}}
Our proof is an application of Proposition~\ref{mdpmdep}.\\

\begin{align*}
V_{\ell,1}^{n}(g,h)-V_{\ell,1}(g,h)&=\frac{1}{n}\sum_{i=1}^{n}g\left(\sqrt n \Delta X_{\ell}(I^{i})\right)h\left(\sqrt n\Delta X_{\ell}(I^{i+1})\right)-\int_0^T \Sigma_{\sigma_{l,u}}(g)\Sigma_{\sigma_{l,u}}(h)\mathrm du\\
&:= \frac{1}{n}\sum_{i=1}^{n}Z_i+L_n,
\end{align*}
where 
$$Z_i=g\left(\sqrt n \Delta X_{\ell}(I^{i})\right)h\left(\sqrt n\Delta X_{\ell}(I^{i+1})\right)-\mathbb{E}\biggl(g(\sqrt n \Delta X_{\ell}(I^{i}))h(\sqrt n \Delta X_{\ell}(I^{i+1}))\biggr),$$
and
$$
L_n= \frac{1}{n}\sum_{i=1}^{n}\biggl[\mathbb{E}\biggl(g(\sqrt n \Delta X_{\ell}(I^{i}))h(\sqrt n \Delta X_{\ell}(I^{i+1}))\biggr)-\int_0^T \Sigma_{\sigma_{l,u}}(g)\Sigma_{\sigma_{l,u}}(h)\mathrm du\biggr].\\
$$

In fact, by a straightforward calculation, one can see that \cite{DGW}, $\sqrt{n}L_n\longrightarrow 0,$ as $n\rightarrow \infty.$

So the LDP still holds with $V_{\ell,1}(g,h)$ substitude by $\mathbb{E}V_{\ell,1}^{n}(g,h)$.

Hence, by the $1$-dependence of $(Z_i)$ and from (\ref{eq_bv1}) together the Proposition~\ref{mdpmdep}, the Proposition~\ref{mdp_bv1} holds.
\quad$\Box$

\section{Appendix A}\label{appendix}

Let us begin with some few bibliographical notes on the MDP. Borovkov and Mogulskii (\cite{BM1}, \cite{BM2}) considered the MDP for Banach valued i.i.d.r.v. sequences $(Z_n)_{n\geq 1}$. Under the condition that $\mathbb E e^{\delta |Z_1|}<+\infty,$ for some $\delta >0$, they proved the MDP for $\sum_{i=1}^{n}Z_i$.\par
For $b_n=n^{\alpha}$ with $0<\alpha<\frac{1}{2}$, Chen \cite{CX} found the necessary and sufficient condition for the MDP in a Banach space, and he obtained the lower bound for general $b_n$ under very weak condition. Using the isoperimetry techniques, Ledoux \cite{LM}(see also \cite{EL}) obtained the necessary and sufficient condition for the general sequence $b_n$ satisfying:
\begin{equation}\label{speed}b_n\longrightarrow \infty \quad{\text{and}}\quad \frac{b_n}{\sqrt n}\longrightarrow 0.\end{equation}

The results of Ledoux \cite{LM} is extended to functional empirical processes (in the setting of non parametrical statistics) by Wu \cite{LW2}. The further developments are given by Dembo and Zeitouni \cite{DZ}. Arcones \cite{AM} obtains the MDP of functional type with the following condition

\begin{equation}
\limsup_{n\rightarrow \infty}\frac{1}{b_n^2}\log\bigg(n \mathbb P\big(|Z_1|>b_n\sqrt{n}\big)\bigg)=-\infty.
\label{A}
\end{equation}

\vspace{5pt}

How to extend the MDP to the  dependent situations has recently attracted much attention and remarkable works. The Markovian case has been studied under successively less restrictive conditions (see Wu \cite{LW2}, Chen \cite{CX} for the relevant references) and very recently under weak conditions by de Acosta \cite{DA} and Chen \cite{CX} for the lower bound (under different and non comparable conditions) and by de Acosta-Chen \cite{AC} and Chen \cite{CX} (under the same condition but different proof) for the upper bound. Guillin \cite{G1} obtained uniform (in time) MDP for  functional empirical processes. Using regeneration split chain method, Djellout and Guillin (\cite{DG}, \cite{DG1}) extend the characterization of MDP for i.i.d.r.v. case of Ledoux \cite{LM} to Markov chains. The geometric ergodicity is substituted by a tail on the first time of return to the atom. Their conditions are weaker than Theorem 1 in \cite{AC}, which allow them to obtain MDP for empirical measures and functional empirical processes.\par

For the studies of the MDP of martingale see  Puhalskii \cite{Puha}, Dembo \cite{Dembo}, Gao \cite{Gao}, Worms \cite{Worms} and Djellout \cite{DH}.

\vspace{5pt}

Those works motivate directly the studies here. Our main aim is to prove the Chen-Ledoux type theorem for the MDP of a sequence of $m$-dependent random variables.

\subsection{Moderate deviations for $m$-dependent randon variables}

Let $(b_n)$ a sequence of positive real numbers satisfying (\ref{speed}).

Let $X_n$ be an $m$-dependent random variables, namely, for every $k\ge 1$ the following two collections
$$\{X_1,\cdots,X_k\}\qquad{\text {and}} \qquad\{X_{k+m+1},X_{k+m+2},\cdots\}$$
are independent. In this terminology, an independent identically distributed (i.i.d.) random variables sequence is $0$- dependent and, for any non-negative intergers $m_1<m_2$, $m_1$ dependence implies $m_2$-dependence. The aim of this section is to consider the MDP for $m$ dependent random variables sequences which are not necessarily stationnary under the Chen-Ledoux type condition. This condition links the speed of the MDP to the queue of the distribution of the random variable.

\begin{prop}\label{mdpmdep}
Let $(X_k)$ be a sequence of $m$-dependent random variables with zero means such that 
\begin{equation}\label{Chen-Ledoux}
\limsup_{n\rightarrow \infty}\frac{1}{b_n^2}\log\left(n\max_{i=1}^n \mathbb P(|X_i|>b_n\sqrt{n})\right)=-\infty.
\end{equation}
Write $S_n=\sum_{j=1}^nX_j$ and let $(b_n)$ a sequence of positive numbers satisfying  (\ref{speed}).
Then $\displaystyle \left(\frac{S_n}{b_n\sqrt{n}}\right)_{n\ge 1}$ satisfies the LDP on $\mathbb R$ with speed $b_n^2$ and rate fonction $I(x)=x^2/2\sigma^2,$ where $\displaystyle \sigma^2:=\lim_{n\rightarrow \infty}\frac 1n {\rm Var}(S_n)<\infty.$

\end{prop}

\begin{rem}
Remark  that the condition (\ref{Chen-Ledoux}) implies that
$$\lim_{n\rightarrow \infty}\sup_{1\le \ell\le n}\mathbb E(X^2_{\ell})<\infty.$$
This is a little adaptation of the proof of the Lemma 2.5 in \cite{EL}.
\end{rem}

\begin{rem} \label{exists} If there exists some $\delta>0$ such that 
$$\lim_{n\rightarrow\infty}\sup_{1\le \ell\le n  }\mathbb E (e^{\delta |X_{\ell}|})<\infty,$$
so the condition (\ref{Chen-Ledoux}) is satisfied for every sequence $b_n$. See Remark 2.32 in  \cite{EL}.
\end{rem}

\begin{rem}See \cite{DG}. If $b_n=n^{\frac{1}{p}-\frac{1}{2}}$ with  $1<p<2$ so the condition (\ref{Chen-Ledoux}) is equivalent to
$$\lim_{n\rightarrow\infty}\sup_{1\le \ell\le n  }\mathbb E (e^{\delta |X_{\ell}|^{2-p}})<\infty.$$

\end{rem}

\begin{lem}\label{mdpind}Let $(Y_k)$ be a sequence of independent random variables with zero means such that  the condition (\ref{Chen-Ledoux}) is satisfied. Write $S'_n=\sum_{j=1}^nY_j$ and let $(b_n)$ a sequence of positive numbers satisfying  (\ref{speed}). Then $\displaystyle \left(\frac{S'_n}{b_n\sqrt{n}}\right)_{n\ge1}$ satisfies the LDP on $\mathbb R$ with speed $b_n^2$ and rate fonction $I(x)=x^2/2\sigma^2,$ where $\displaystyle\sigma^2:=\lim_{n\rightarrow \infty}\frac 1n {\rm Var}(S_n)<\infty.$
\end{lem}

{\textbf{Proof of Lemma \ref{mdpind}}} The proof of this result is a little adaptation of the proof of the independent and identically distributed random variables given in \cite{DG}. $\Box$

\vspace{15pt}

{\textbf{Proof of Proposition \ref{mdpmdep}}} 
We will do the proof in the case $m=1$ for simplicity.

Fix the integer $p>1$ and write, for each $n\ge 1$, $$n=k_np+r_n$$
where $k_n$ and $r_n$ are non-negative integers with $0\le r_n\le p-1.$ Define
$$Y_k=\sum_{(k-1)p<j<kp}X_j,\qquad k=1,2,\cdots$$
Then $(Y_k)_{k\ge 1}$ is an independent random variables and

Notice that 
$$S_n=\sum_{\ell=1}^{k_n}Y_{\ell}+\sum_{\ell =1}^{k_n}X_{\ell p}+\sum_{\ell=k_np+1}^nX_j \qquad (n\ge 1).$$

We have to prove that 
\begin{equation}\label{N1}
\limsup_{n\rightarrow \infty}\frac{1}{b_n^2}\log{\mathbb P}\left(\left|\sum_{\ell= 1}^{k_n}X_{\ell p}\right| >\delta b_n\sqrt{n}\right)=-\infty.
\end{equation}

Since the random variables $(X_{\ell p})_{\ell \ge 1}$ are independent, by Lemma \ref{mdpind}, we use the MDP for $\displaystyle\left(\frac{1}{b_n\sqrt{n}}\sum_{\ell=1}^{k_n}X_{\ell p}\right)_{\ell \ge 1}$, to obtain that for $n$ large 

$$\frac{1}{b_n^2}\log{\mathbb P}\left(\left|\sum_{\ell= 1}^{k_n}X_{\ell p}\right| >\delta b_n\sqrt{n}\right)\le -\frac{\delta^2}{2\sigma^2_0},$$
where 
$$\sigma^2_0=\lim_{n\rightarrow \infty}\frac{1}{n}\sum_{\ell=1}^{k_n}{\rm Var}(X_{\ell p})\le C \lim_{n\rightarrow \infty}\frac{k_n}{n}=\frac{C}{p}.$$

So
$$\limsup_{n\rightarrow \infty}\frac{1}{b_n^2}\log{\mathbb P}\left(\left|\sum_{\ell= 1}^{k_n}X_{\ell p}\right| >\delta b_n\sqrt{n}\right)\le -\frac{p}{2}\frac{\delta^2}{C}.$$

Letting $p$ goes to infinty, we get (\ref{N1}).

Now we have to  prove that
\begin{equation}\label{N2}
\limsup_{n\rightarrow \infty}\frac{1}{b_n^2}\log{\mathbb P}\left(\left|\sum_{\ell= k_np+1}^{n}X_{\ell }\right| >\delta b_n\sqrt{n}\right)=-\infty.
\end{equation}

If $r_n=0$ there is no thing to prove, so we suppose that $r_n>0$.
\begin{eqnarray*}
{\mathbb P}\left(\left|\sum_{\ell= k_np+1}^{n}X_{\ell}\right| >\delta b_n\sqrt{n}\right)&\le& \sum_{\ell =k_np+1}^n\mathbb P\left(|X_{\ell}|>\frac{\delta b_n\sqrt{n}}{r_n}\right)\\
&\le& r_n \max_{\ell =k_np+1}^n{\mathbb P}\left(|X_{\ell}|>\delta \frac{b_n\sqrt{n}}{r_n}\right)\\
&\le& n  \max_{\ell =1}^n{\mathbb P}\left(|X_{\ell}|>\delta \frac{b_n\sqrt{n}}{p}\right).
\end{eqnarray*}

Taking the $\log$ and normalizing by $1/b_n^2$,  we obtain (\ref{N2}).
\vspace{5pt}

Now, we have to prove that
\begin{equation}\label{sim}
\limsup_{n\rightarrow \infty}\frac{1}{b_n^2}\log\left[k_n\max_{\ell = 1}^{k_n}\mathbb P\left(|Y_{\ell}|>b_n \sqrt{n}\right)\right]=-\infty.
\end{equation}

We have
\begin{eqnarray*}
k_n\max_{\ell = 1}^{k_n}\mathbb P\left(|Y_{\ell}|>b_n \sqrt{n}\right)&\le& k_n \max_{\ell = 1}^{k_n}\sum_{j=(\ell-1)p+1}^{\ell p-1}\mathbb P\left(|X_{j}|>\frac{1}{p}b_n \sqrt{n}\right)\\
&\le&k_np \max_{\ell=1}^{k_n}\max_{j=(\ell-1)p+1}^{\ell p-1} \mathbb P\left(|X_{j}|>\frac{1}{p}b_n \sqrt{n}\right)\\
&\le&n  \max_{\ell=1}^{n}\mathbb P\left(|X_{j}|>\frac{1}{p}b_n \sqrt{n}\right).
\end{eqnarray*}

Taking the $\log$ and normalizing by $1/b_n^2$ and letting $n$ goes to infinity, we obtain (\ref{sim}).

So we deduce that the sequence $\displaystyle \left(\frac{1}{b_n\sqrt{n}}\sum_{\ell=1}^{k_n}Y_{\ell}\right)_{n\ge 1}$ satisfies the LDP with speed $b_n^2$ and rate function $I(x)=x^2/2\sigma_Y^2$, where
\begin{equation}
\sigma_Y^2:=\lim_{n\rightarrow \infty} \frac{1}{n}{\rm Var}\left(\sum_{\ell =1}^{k_n}Y_{\ell}\right).
\end{equation}

Now we have to prove that $\sigma^2=\sigma_Y^2$. For that we have
\begin{eqnarray*}
\Delta_n:={\rm Var}S_n-{\rm Var}\left(\sum_{\ell=1}^{k_n}Y_{\ell}\right)&=&\sum_{\ell=1}^{k_n}{\mathbb E}(X_{\ell p}^2)+2\sum_{\ell =1}^{k_n}\left(\mathbb E X_{\ell p}X_{\ell p+1}+\mathbb E X_{\ell p} X_{\ell p -1}\right)\\
&+&\sum_{\ell=pk_n+1}^{pk_n+r_n}\mathbb E(X_{\ell}^2+X_{\ell}X_{\ell+1}+X_{\ell}X_{\ell-1}).
\end{eqnarray*}

Now since $\sup_{1\le \ell \le n}\mathbb E X_{\ell}^2<\infty,$ we have that $\Delta_n/n\le (k_n C+4 Ck_n+3r_nC)/n$. So $\lim_{n\rightarrow \infty}\Delta_n/n=C/p$, which goes to 0 when $p$ goes to infinity.$\Box$
\vspace{10pt}

\nocite{*}

\bibliographystyle{acm}
\bibliography{bibliosoumismdpF}

\nocite{*}

\vspace{10pt}

\end{document}